\documentclass[reqno,a4paper,12pt]{amsart} 

\usepackage{amsmath,amscd,amsfonts,amssymb}
\usepackage{mathrsfs,dsfont}

\usepackage{bbm}

\allowdisplaybreaks

\numberwithin{equation}{section}
\numberwithin{figure}{section}

\newcommand\C{\mathbb{C}}
\newcommand{\R}{\mathbb{R}}
\newcommand{\T}{\mathbb{T}}
\newcommand{\D}{\mathbb{D}}
\newcommand{\Sp}{\mathbb{S}}
\newcommand{\B}{\mathbb{B}}

\newcommand{\har}{\mathbbm{1}}

\newcommand{\Rl}{\mathrm{Re}\,}

\renewcommand\le{\leqslant}
\renewcommand\ge{\geqslant}
\renewcommand\leq{\leqslant}
\renewcommand\geq{\geqslant}

\theoremstyle{plain}
\newtheorem{thm}{Theorem}[section]

\newtheorem{lemma}[thm]{Lemma}

\newtheorem{prop}[thm]{Proposition}

\newtheorem*{claim*}{Claim}

\theoremstyle{definition}

\newtheorem*{definition*}{Definition}
\newtheorem*{remarks*}{Remarks}
\newtheorem*{remark*}{Remark}

\begin{document}
	
	\title
	[Weighted Chui's conjecture]
	{Weighted Chui's conjecture}
	
	\author{Evgueni Doubtsov}
	\address{
		St. Petersburg Department of Steklov Mathematical Institute, Fontanka 27, St. Petersburg 191023, Russia}
	\email{dubtsov@pdmi.ras.ru}

	\author{Anton Tselishchev}
	\address{
		St. Petersburg Department of Steklov Mathematical Institute, Fontanka 27, St. Petersburg 191023, Russia}
	\email{celis\_anton@pdmi.ras.ru}
	
	\author{Ioann Vasilyev}
	\address{
		St. Petersburg Department of Steklov Mathematical Institute, Fontanka 27, St. Petersburg 191023, Russia}
	\email{ivasilyev@pdmi.ras.ru}

	\subjclass[2020]{30E20}
	\keywords{Coulomb potential, Chui's conjecture,  Cauchy transform, simplest fraction}
	\thanks{A.T. and I.V. are supported by the Foundation for the Advancement of Theoretical Physics and Mathematics ``BASIS''. \\ This research was supported by the Russian Science Foundation (grant No. 23-11-00171, https://rscf.ru/project/23-11-00171/).}

	\begin{abstract}
		The  goals of this paper are threefold. First, we show that a counterpart of the Newman bound related to the Chui conjecture is valid in the case where the gradient of Coulomb potential is generated by arbitrary positive charges placed at the boundary of the unit ball of $\R^d$. Second, we prove that our bound is sharp in the two dimensional case. Finally, we discuss a related problem where the unit charges are placed in the  unit disc.
	\end{abstract}

	\maketitle


	\section{Introduction}
	
	Let $\mathbb D$ denote the unit disc of the complex plane $\mathbb C$. Consider points $z_1, z_2, \ldots, z_n$ on the unit circle $\T=\partial \mathbb D$. The Chui conjecture, formulated in \cite{Chu71}, suggests that the quantity
	\begin{equation}\label{eq:chui}
		\int_{\D}\Big| \sum_{k=1}^n \frac{1}{z-z_k} \Big|\, d m(z),
	\end{equation}
	where $dm(z)$ denotes the Lebesgue measure on the complex plane, attains its minimum if the points $z_1, z_2, \ldots, z_n$ are uniformly distributed on the unit circle, i.e., if $z_k=e^{2\pi i k/n}$. This conjecture has the following natural physical interpretation: in order to minimize the average strength of the electrostatic field generated by $n$ unit charges on the circle, the charges should be distributed uniformly.
	
	Surprisingly, this natural conjecture remains open. Shortly after its formulation, D.J.~Newman~\cite{New72} proved that this average strength of an electrostatic field is bounded from below, i.e., there exists an absolute constant $c > 0$ such that for every $n$ and any points $z_1, z_2, \ldots z_n$ on the unit circle $\T$ the following inequality holds:
	\begin{equation}\label{eq:newbound}
		\int_{\D}\Big| \sum_{k=1}^n \frac{1}{z-z_k} \Big|\, d m(z) \ge c.
	\end{equation}
	To be more precise, the above inequality holds with $c=\pi/18$. This result is, of course, weaker than  Chui's conjecture. Indeed, it is easy to see that if $z_k=e^{2\pi i k/n}$, then the corresponding quantity~\eqref{eq:chui} in fact is bigger than $\pi/18$.

    Notice that even in the case of two charges Chui's conjecture is a non-trivial question, as F. Nazarov's proof in~\cite{MaOvFl} suggests.
	
	A more recent progress was obtained in the paper \cite{ABF21} where the Chui conjecture was proved if one replaces the $L^1$ norm in the quantity~\eqref{eq:chui} with the norm in certain Bergman spaces.
	
	Notice also that Chui's conjecture is closely related to the topic in Approximation Theory called the approximation by simplest fractions, see for instance papers \cite{BS23}, \cite{Bor16}, \cite{Kom23} and \cite{CZ23}.

\medskip
    
	In the present paper we study the following natural questions: what if we consider not unit but $n$ different charges (of the same sign)? Also, what happens if we place them on the unit sphere of $\R^d$ (from the physical point of view, the case $d=3$ is the most interesting)? Clearly, the methods of Complex Analysis have certain limitations in the latter case.
	
	It seems to be difficult to even formulate the analogue of  Chui's conjecture in this case since it is not clear how the optimal distribution of the charges on the unit sphere should look like. However, we are able to formulate and prove an analogue of the Newman bound in this case. We denote by $\Sp^{d-1}$ and $\B^d$ the unit sphere and the unit ball in $\R^d$, respectively.

	\begin{thm}\label{thm:lowerbound}
		Let $d\ge 2$, $x_1, \dots, x_n\in \Sp^{d-1}$ and $\alpha_1, \dots, \alpha_n >0$. Then 
		\begin{equation}\label{eq:lowerbound}
			\int_{\mathbb{B}^d} \left|\sum_{k=1}^n \alpha_k \frac{x_k - x}{|x_k - x|^d} \right|\, dm(x) \ge c_d 
			\frac{\sum_{k=1}^n\alpha_k^{1+\frac{2}{d}}}{\sum_{k=1}^n\alpha_k^{\frac{2}{d}}},
		\end{equation}
		where $c_d >0$ depends only on $d$, and $dm$ stands for the Lebesgue measure on $\mathbb R^d$.
	\end{thm}
	Passing to complex coordinates, it is easy to see that if $d=2$ and $\alpha_1=\alpha_2=\ldots=\alpha_n=1$, then this estimate coincides with Newman's bound \eqref{eq:newbound}. 

    Similar problems were considered by J. Korevaar and his coauthors, see for instance papers~\cite{Kor64} and~\cite{KM01}.
	
	We also note that the estimate~\eqref{eq:lowerbound} in the special case $d=3$ and unit values of $\alpha_k$ was posed as an open problem in \cite{Arr24}. Once again, we mention that this case has a clear physical interpretation: if $x_1, \ldots, x_n\in \Sp^{2}$, then the electrostatic field in the unit ball $\B^3$ corresponding to the Coulomb potential
	$$
	U(x) = \sum_{k=1}^n \frac{1}{|x-x_k|}
	$$
	is given by 
	$$
	\nabla U(x) = -\sum_{k=1}^n \frac{x-x_k}{|x-x_k|^3}.
	$$
	Therefore, our result states that the mean value of this electrostatic field is not too small, i.e., the charges cannot ``compensate each other''.

\medskip
    
Let $\mu$ be a measure on the closed unit disc $\overline{\mathbb D}$. Recall that its Cauchy transform is defined for $z\in \mathbb C$ with $|z|<1$ by
$$\mathcal{C}\mu(z):=\int \frac{d\mu(\xi)}{\xi-z}.$$
Suppose now that 
$$\nu=\sum_{k=1}^n \alpha_k \delta_{z_k},$$ 
with $\alpha_k$ being all positive  and $z_k\in \mathbb T$, is a discrete measure. Then
$$\mathcal{C}\nu(z)= \sum_{k=1}^n \frac{\alpha_k}{z_k-z}.$$
Notice that our result in the two dimensional case gives rise to the following interesting lower bound.
	\begin{thm}\label{cauchy}
	In the notation above we have
	$$\|\mathcal{C}\nu\|_{L^1(\mathbb D)}\geq C\frac{\sum_{k=1}^n\alpha_k^2}{\|\nu\|},$$
	where $C$ is an absolute constant and $\|\ldots\|$ denotes the total variation of a measure. 
	\end{thm}
	
	It is natural to ask whether the estimate in Theorem~\ref{thm:lowerbound} is optimal. This question is non-trivial in our case since we need to carefully choose the positions of the points $x_1, \dots, x_n$ on the unit sphere in order to get as many cancellations as possible in the integral on the right hand side of the formula~\eqref{eq:lowerbound}. We will prove that our estimate is optimal only in the case where $d=2$.
	
	\begin{thm}\label{thm:upperbound_d=2}
		For every $n\in\mathbb{N}$ and any positive numbers $\alpha_1, \ldots, \alpha_n$ there exist points $z_1, \ldots z_n$ on the unit circle $\T$ such that
		$$
		\int_{\D}\Big| \sum_{k=1}^n \frac{\alpha_k}{z-z_k} \Big|\, d m(z)\le C \frac{\sum_{k=1}^n\alpha_k^2}{\sum_{k=1}^n \alpha_k}.
		$$
		where $C > 0$ is an absolute constant.
	\end{thm}
	
	Notice that the positivity assumption on $\alpha_k$ in Theorem~\ref{thm:lowerbound} is essential, as the following result shows.
\begin{prop}
\label{lem1}
Let $a\in \overline{\mathbb D}$ and $b\in \overline{\mathbb D}$. Suppose that $\delta=|a-b|\leq 1$. Then
$$
\int_{\D}\left|\frac{1}{z-a}-\frac{1}{z-b}\right|dm(z)\lesssim \delta+\delta \log\Bigl(\frac{1}{\delta}\Bigr).$$
\end{prop}
Indeed, to see that the positivity of $\alpha_k$ is needed, it suffices to choose $a=1$ and $b=e^{i\varepsilon}$ with sufficiently small $\varepsilon>0$  in Proposition~\ref{lem1}.
	

    Here and everywhere below notation $X\lesssim Y$ means that $X\le C\cdot Y$ for some absolute constant $C > 0$.
	
	\section{Newman's bound for arbitrary charges}
	
	In this section we prove Theorem~\ref{thm:lowerbound}. The main idea of the proof is to consider appropriate orthogonal projections on 2-dimensional subspaces of $\R^d$ and then perform the computations similar to what is done in \cite{New72}.

	\subsection{Auxiliary lemmas} We begin the proof with three simple technical lemmas.
	
	\begin{lemma}\label{l_1}
		Let $y\in \Sp^{d-1}$ and $x\in \B^d$. Then
		\[
		\langle \frac{y-x}{|y-x|^d}, x\rangle \ge -\frac{1}{2|y-x|^{d-2}}
		\]
	\end{lemma}
	\begin{proof}
		Consider the 2-dimensional plane containing $0$, $x$ and $y$ and introduce complex coordinates in it.
		Given $u,v\in\C$, recall that $\langle u, v \rangle = \Rl(u \overline{v})=\Rl (\overline{u}v)$.
		
		We have to prove that
		\[
		\langle \frac{y-x}{|y-x|^{2}}, x\rangle  \ge -\frac{1}{2}.
		\]
		Replacing $x$ and $y$ with their complex coordinates $z$ and $w$ (where $|z| < 1$, $|w|=1$), we get that this inequality is equivalent to
		\[
		\Rl \left( \frac{z}{w-z}\right) = \Rl \left( \frac{\overline{w} -\overline{z}}{|w-z|^2}z\right) \ge -\frac{1}{2}.
		\]
		The above estimate is equivalent to
		\[
		\Rl \left( \frac{w+ z}{w-z}\right) \ge 0,
		\]
		which is a well-known property of the Poisson kernel.
	\end{proof}
	
	The next lemma shows that for certain $x\in\B^d$ we in fact have a better estimate.
	
	\begin{lemma}\label{l_2}
		Let $y\in \Sp^{d-1}$, $0<r<\frac{1}{2}$, and $Q$ be the ball of radius $r$ tangent from inside
		to $\Sp^{d-1}$ at the point $y$. Then for every $x\in Q$ we have
		\[
		\langle \frac{y-x}{|y-x|^d}, x\rangle + \frac{1}{2|y-x|^{d-2}} \ge\frac{1-r}{2r}\frac{1}{|y-x|^{d-2}}.
		\]
	\end{lemma}
	\begin{proof}
		As in the proof of Lemma~\ref{l_1}, we replace $x$ and $y$ by $z$ and $w$, respectively.
		Therefore, it suffices to prove that
		\[
		\Rl \left( \frac{w+ z}{w-z}\right) \ge \frac{1-r}{r},\quad z, w\in\C, \ |w|=1, \ z\in \mathcal Q_2,
		\]
		where $\mathcal Q_2$ denotes the disc of radius $r$ tangent from inside to the unit disc $\D$ at $w$.
		Standard computations guarantee that the set
		\[
		\left\{ z\in\C: \Rl \left( \frac{w+ z}{w-z}\right) \ge c \right\}
		\]
		is the disc of radius $\frac{1}{1+c}$ tangent to $\D$ at $w$.
		It remains to observe that $r=\frac{1}{1+c}$ means $c= \frac{1-r}{r}$.
	\end{proof}
	
	Our final lemma is a simple geometric observation.
	
	\begin{lemma}\label{l_3}
		Let $y_1, y_2 \in \Sp^{d-1}$, $0< r_1, r_2 < \frac{1}{2}$.
		Let $B_i$, $i=1,2$, denote the ball of radius $r_i$ tangent from inside to $\Sp^{d-1}$ at $y_i$.
		Assume that $x\in B_1$ and $x\notin B_2$.
		Then
		\[
		\frac{|x- y_1|}{|x- y_2|} \le \left(\frac{r_1}{r_2} \cdot \frac{1-r_2}{1-r_1}  \right)^\frac{1}{2}.
		\]
	\end{lemma}
	\begin{proof}
		We have $x\in B_1$ if and only if $|x- (1-r_1)y_1|^2< r_1^2$, that is,
		\begin{equation}\label{e_l3_1}
			|x|^2 - 2 (1-r_1)\langle x, y_1 \rangle +1 - 2r_1 <0.
		\end{equation}
		Rewrite the above estimate as
		\begin{equation}\label{e_l3_2}
			|x- y_1|^2 \le 2r_1(1 - \langle x, y_1 \rangle).
		\end{equation}
		By \eqref{e_l3_1},
		\begin{equation}\label{e_l3_3}
			1-\langle x, y_1 \rangle \le \frac{1-|x|^2}{2-2r_1}.
		\end{equation}
		By \eqref{e_l3_2} and \eqref{e_l3_3},
		\begin{equation}\label{e_l3_4}
			|x- y_1|^2 \le \frac{r_1}{1-r_1}(1-|x|^2).
		\end{equation}
		
		Analogously, we have $x\notin B_2$ if and only if $|x|^2 - 2 (1-r_2)\langle x, y_2 \rangle +1 - 2r_2 \ge 0$,
		and we deduce that
		\begin{equation}\label{e_l3_5}
			|x- y_2|^2 \ge \frac{r_2}{1-r_2}(1-|x|^2).
		\end{equation}
		Combining \eqref{e_l3_4} and \eqref{e_l3_5}, we obtain the conclusion of the lemma.
	\end{proof}
	
	\subsection{Proof of Theorem~\ref{thm:lowerbound}} Now we apply the above lemmas in order to prove Theorem~\ref{thm:lowerbound}.
	
	Put
	\[
	G = \sum_{j=1}^n \alpha_j^{\frac{2}{d}}, \qquad r_k =\frac{1}{2^{d+2}} \frac{\alpha_k^{\frac{2}{d}}}{G}, \; k=1,\ldots,n.
	\]
	Let $Q_k$ denote the ball of radius $r_k$ tangent from inside to $\Sp^{d-1}$ at $x_k$.
	Put $Q = \cup_1^n Q_k$ and let $I$ denote the integral which we need to estimate, i.e., 
	$$
	I = 	\int_{\mathbb{B}^d} \left|\sum_{k=1}^n \alpha_k \frac{x_k - x}{|x_k - x|^d} \right|\, dm(x).
	$$
	Then
	\[
	\begin{split}
		I 
		&\ge \int_{Q} \left|\sum_{k=1}^n \langle \alpha_k \frac{x_k - x}{|x_k - x|^d}, x \rangle \right|\, dm(x) \\
		&\ge \int_{Q} \left\{ \left|\sum_{k=1}^n \left(\langle \alpha_k \frac{x_k - x}{|x_k - x|^d}, x \rangle 
		+ \frac{\alpha_k}{2|x_k - x|^{d-2}}\right) \right|
		- \left|\sum_{k=1}^n \frac{\alpha_k}{2|x_k - x|^{d-2}} \right|\right\}\, dm(x)
	\end{split}
	\]
	Now, we apply Lemma~\ref{l_2} for $x\in Q_k$ and Lemma~\ref{l_1} for $x\notin Q_k$ and conclude that
	\[
	\begin{split}
		\alpha_k\langle \frac{x_k - x}{|x_k - x|^d}, x \rangle 
		+  \frac{\alpha_k}{2|x_k - x|^{d-2}}
		&\ge \frac{1-r_k}{2r_k} \frac{\alpha_k}{|x_k - x|^{d-2}} \har_{Q_k}(x) \\
		&\ge \frac{1}{4r_k} \frac{\alpha_k}{|x_k - x|^{d-2}} \har_{Q_k}(x),
	\end{split}
	\]
	since $r_k < \frac{1}{2}$. Note that by the definition of $r_k$ the latter quantity equals
	$$
	2^d {\alpha_k}^{1-\frac{2}{d}} G \frac{1}{|x_k - x|^{d-2}} \har_{Q_k}(x).
	$$
    
	Therefore,
	\[
	\begin{split}
		I 
		&\ge \int_{Q} \left\{ \left|\sum_{k=1}^n 2^d  G \frac{{\alpha_k}^{1-\frac{2}{d}}}{|x_k - x|^{d-2}} \har_{Q_k}(x) \right|
		- \left|\sum_{k=1}^n \frac{\alpha_k}{2|x_k - x|^{d-2}} \right|\right\}\, dm(x) \\
		&= 
		\int_{Q} \left( \sum_{k=1}^n 2^{d-1}  G \frac{{\alpha_k}^{1-\frac{2}{d}}}{|x_k - x|^{d-2}} \har_{Q_k}(x)
		\right)\, dm(x) \\
		&+ \int_{Q} \left( \sum_{k=1}^n 2^{d-1}  G \frac{{\alpha_k}^{1-\frac{2}{d}}}{|x_k - x|^{d-2}} \har_{Q_k}(x)
		- \sum_{j=1}^n \frac{\alpha_j}{2|x_j - x|^{d-2}}\right)\, dm(x)
		=: A + B.
	\end{split}
	\]
	
	Since $|x-x_k| \le r_k$ for $x\in Q_k$, we have
	\[
	\begin{split}
		A &= \sum_{k=1}^n\int_{Q_k}   2^{d-1}  G \frac{{\alpha_k}^{1-\frac{2}{d}}}{|x_k - x|^{d-2}}
		\, dm(x)\\
		&\ge \sum_{k=1}^n m(Q_k) 2^{d-1}  G {\alpha_k}^{1-\frac{2}{d}} r_k^{2-d}
		\\ &=  \sum_{k=1}^n\frac{m(Q_k)}{r_k^d}\cdot \frac{\alpha_k^{\frac{4}{d}}}{2^{2d+4}G^2}\cdot 2^{d-1}G {\alpha_k}^{1-\frac{2}{d}} \\
		&= c_d \frac{\sum_{k=1}^n\alpha_k^{1+\frac{2}{d}}}{\sum_{k=1}^n\alpha_k^{\frac{2}{d}}}.
	\end{split}
	\]
	Thus, to prove the required estimate, we have to show that $B \ge 0$.
	Therefore, it suffices to prove that
	\begin{equation}\label{e_B_0}
		\sum_{k=1}^n 2^{d-1}  G \frac{{\alpha_k}^{1-\frac{2}{d}}}{|x_k - x|^{d-2}} \har_{Q_k}(x)
		\ge \sum_{j=1}^n \frac{\alpha_j}{2|x_j - x|^{d-2}}, \quad x \in Q.
	\end{equation}
	Fix $x\in Q$ and put $E_x = \{j: x\in Q_j\}\neq\varnothing$.
	Select $k\in E_x$ such that
	\begin{equation}\label{e_min}
		\frac{|x_k - x|^{2}}{r_k} \le \frac{|x_j - x|^{2}}{r_j} \quad\textrm{for all}\ j \in E_x.
	\end{equation}
	We claim that
	\begin{equation}\label{e_star}
		2^d \frac{\alpha_k^{1-\frac{2}{d}}}{|x_k - x|^{d-2}} \ge \frac{\alpha_j^{1-\frac{2}{d}}}{|x_j - x|^{d-2}}
	\end{equation}
	Observe that \eqref{e_star} guarantees that
	\[
	\sum_{j=1}^n 2^{d-1} \alpha_j^\frac{2}{d}\frac{\alpha_k^{1-\frac{2}{d}}}{|x_k - x|^{d-2}}
	\ge \sum_{j=1}^n \frac{\alpha_j}{2|x_j - x|^{d-2}}.
	\]
	The above estimate can be rewritten as
	$$
	2^{d-1}  G \frac{{\alpha_k}^{1-\frac{2}{d}}}{|x_k - x|^{d-2}} \ge \sum_{j=1}^n \frac{\alpha_j}{2|x_j - x|^{d-2}},
	$$ 
	and since $\har_{Q_k}(x) = 1$, it implies \eqref{e_B_0}. Hence, in order to finish the argument, we need to prove \eqref{e_star}.
	
	We consider two cases. Firstly, assume that $x\in Q_j$. Then \eqref{e_min} holds and we obtain
	\[
	\frac{\alpha_k^{\frac{2}{d}}}{|x_k - x|^{2}} \ge \frac{\alpha_j^{\frac{2}{d}}}{|x_j - x|^{2}},
	\]
	which implies \eqref{e_star}.
	
	Secondly, assume that $x\notin Q_j$. Then, by Lemma~\ref{l_3},
	\[
	\frac{|x_k - x|}{|x_j - x|} \le \left( \frac{r_k}{r_j} \frac{1-r_j}{1-r_k} \right)^\frac{1}{2}
	\le 2\left(\frac{r_k}{r_j} \right)^\frac{1}{2} = 2 \frac{\alpha_k^\frac{1}{d}}{\alpha_j^\frac{1}{d}}.
	\]
	The above estimate implies \eqref{e_star}.
	This finishes the proof of the theorem.
	
	\section{Optimality of the estimate in the case $d=2$}
    		\subsection{Reduction to an estimate of a single fraction}
	
	We now concentrate on the proof of Theorem~\ref{thm:upperbound_d=2}. To this end, fix numbers $\alpha_k>0$ for $1\leq k\leq n$. Put $$A=\sum_{k=1}^n\alpha_k, \qquad B=\sum_{k=1}^n\alpha_k^2, \qquad  l_k=\frac{2\pi\alpha_k}{A}.$$ Then $\sum_{k=1}^n l_k=2\pi$. Now consider pairwise disjoint semi-closed intervals $I_k\subset [-\pi,\pi)$, $k=1, \ldots, n$  such that $|I_k|=l_k$ and  $\bigcup_{k=1}^n I_k=[-\pi,\pi)$. We choose $z_k=e^{i\theta_k}$ where $\theta_k$ is the middle of the interval $I_k$.
	
	We need to prove that
		\begin{equation}\label{sharp}
			\int_{\mathbb D} |F(z)| dm(z)\lesssim \frac{B}{A},
		\end{equation}
		where we denote for $z\in \mathbb D$ as follows:
        $$F(z):=\sum_{k=1}^n \frac{\alpha_k}{(z-z_k)}.$$

		The main idea of the proof is to use the following identity:
		$$
		\int_{-\pi}^\pi\frac{d\theta}{z-e^{i\theta}}=0,
		$$
		valid for all $|z|<1$. To see that this is true it suffices first to write 
		$$\frac{1}{(z-e^{i\theta})}=-e^{-i\theta}(1-ze^{-i\theta})^{-1}=-e^{-i\theta}\sum_{j=0}^\infty (ze^{-i\theta})^j,$$ 
		where the series in question converges since $|z|<1$, and further integrate this series term-wise.
Our choice to subtract the mean of the Cauchy kernel is inspired by the well-known idea in the theory of singular integral operators, see for instance~\cite{Vas20}.
		
		Therefore, we write
		\[
		\begin{split}
			\int_{\mathbb D} |F(z)| dm(z) =& \int_{\mathbb D} \Bigl|F(z) -\frac{A}{2\pi}\int_{-\pi}^\pi\frac{d\theta}{z-e^{i\theta}}\Bigr| dm(z)\\
			= 
			&\int_{\mathbb D} \Bigl|\sum_{k=1}^n \frac{\alpha_k}{z-z_k} -\frac{A}{2\pi} \sum_{k=1}^n \int_{I_k}\frac{d\theta}{z-e^{i\theta}}\Bigr| dm(z) \\
			&\leq \frac{A}{2\pi} \sum_{k=1}^n \int_{\mathbb D} \Bigl|\frac{l_k}{z-z_k} - \int_{I_k}\frac{d\theta}{z-e^{i\theta}}\Bigr| dm(x) \\
			&= \frac{A}{2\pi} \sum_{k=1}^n l_k\int_{\mathbb D} \Bigl|\frac{1}{z-z_k} - \frac{1}{l_k}\int_{I_k}\frac{d\theta}{z-e^{i\theta}}\Bigr| dm(x).
		\end{split}
		\]
		Notice that we are done once we prove the estimate
		\begin{equation}\label{main}
			\int_{\mathbb D} \Bigl|\frac{1}{z-z_k} -\frac{1}{l_k}\int_{I_k} \frac{d\theta}{z-e^{i\theta}}\Bigr| dm(z)\lesssim l_k,
		\end{equation}
		since in this case the quantity $\int_{\D}|F(z)|\, dm(z)$ will be estimated by
		$$
		\frac{A}{2\pi}\sum_{k=1}^n l_k^2 \lesssim A\sum_{k=1}^n \frac{\alpha_k^2}{A^2}=\frac{B}{A}.
		$$

        		\subsection{Estimate of a single simple fraction}
		It remains now only to prove the following statement.
		\begin{lemma}
			Let $I\subset [-\pi,\pi)$ be a semi-closed interval. Then
			\begin{equation}\label{auxlem}
				\int_{\mathbb D} \Bigl|\frac{1}{z-w} -\frac{1}{l}\int_{I} \frac{d\theta}{z-e^{i\theta}}\Bigr| dm(z)\lesssim l,
			\end{equation}
			where $w=e^{i\theta_0}$ with $\theta_0$ being the middle of $I$ and $l$ is the length of $I$.
		\end{lemma}
		\begin{proof}
			Without loss of generality we assume that $w=1$ and thus $I=[-l/2,l/2)$. We split the exterior integral over $\mathbb D$ into two parts: a neighborhood of the pole $1$ and its compliment. 
			
			We first treat the  neighborhood term. Integration in polar coordinates and the Fubini theorem yield
			\[
			\begin{split}
				I_1&:=\int_{|z-1|\leq 4l} \Bigl|\frac{1}{z-1} - \frac{1}{l} \int_{-l/2}^{l/2} \frac{d\theta}{z-e^{i\theta}}\Bigr| dm(z) \\
				&\leq\int_{|z-1|\leq 4l} \frac{dm(z)}{|z-1|} + \int_{|z-1|\leq 4l}\frac{1}{l} \int_{-l/2}^{l/2} \frac{d\theta}{|z-e^{i\theta}|}dm(z)\\
				&\leq l + \frac{1}{l} \int_{-l/2}^{l/2}\int_{|z-1|\leq 4l} \frac{dm(z)}{|z-e^{i\theta}|}d\theta.
			\end{split}
			\]
			Notice that $\{z:|z-1|\leq 4l\}\subset \{z:|z-e^{i\theta}|\leq 10l\}$. Hence
			$$
			I_1\lesssim l+ \frac{1}{l} \int_{-l/2}^{l/2}\int_{|z-e^{i\theta}|\leq 10l} \frac{dm(z)}{|z-e^{i\theta}|}d\theta\lesssim l,
			$$
			where in the last inequality we simply used the integration in polar coordinates once again.
			
			Now we proceed to the integral over the compliment of the neighborhood of the pole. We first observe that
			$$
			\frac{1}{l} \int_{-l/2}^{l/2} \frac{d\theta}{z-e^{i\theta}}=\frac{1}{l}\int_{0}^{l/2} \Big( \frac{1}{z-e^{i\theta}} + \frac{1}{z-e^{-i\theta}} \Big)\, d\theta = \frac{1}{l} \int_0^{l/2}\frac{2z-2\cos\theta}{(z-e^{i\theta})(z-e^{-i\theta})}d\theta.
			$$
			Hence we can write as follows:
			\[
			\begin{split}
				\biggl|\frac{1}{z-1} &- \frac{1}{l} \int_{-l/2}^{l/2} \frac{d\theta}{z-e^{i\theta}}\biggr| \\
                &=\biggl|\frac{2}{l}\int_0^{l/2}\frac{1}{z-1}-\frac{z-\cos\theta}{(z-e^{i\theta})(z-e^{-i\theta})}d\theta\biggr|\\
				&=\biggl|\frac{2}{l}\int_0^{l/2}\frac{(1-\cos(\theta))(z+1)}{(z-1)(z-e^{i\theta})(z-e^{-i\theta})}d\theta\biggr|\\
                &\lesssim \frac{1}{l}\int_0^{l/2}\frac{\theta^2}{|z-1||z-e^{i\theta}||z-e^{-i\theta}|}d\theta \lesssim \frac{l^2}{|z-1|^3},
			\end{split}
			\]
			where in the last inequality we made use of the facts that $|z-1|\lesssim|z-e^{i\theta}|$ and that $|z-1|\lesssim|z-e^{-i\theta}|$ which in turn is true since $|z-1|>4l$ and $|e^{\pm i\theta}-1|\le l$.
			
			We are now ready to conclude as follows
			\[
			\begin{split}
				I_2&:=\int_{|z-1|\geq 4l}\biggl|\frac{1}{z-1} - \frac{1}{l} \int_{-l/2}^{l/2} \frac{d\theta}{z-e^{i\theta}}\biggr|dm(z) \\
                &\lesssim l^2\int_{|z-1|\geq 4l}\frac{dm(z)}{|z-1|^3}\\
				&\lesssim l^2\int_{4l}^\infty\frac{r dr}{r^3}\lesssim l.
			\end{split}
			\]
			It remains to collect the estimates and the proof of the lemma is finished.
		\end{proof}

        	\section{Concluding remarks and open problems}\label{s_remarks}

We finish the paper by proving Proposition 1.4 and posing several open questions.

		\subsection{Proof of Proposition~\ref{lem1}} 
        \begin{proof}
Denote $U_a=\{z\in\D: |z-a|<\delta\} \; \text{ and } \; U_b=\{z\in\D: |z-b|<\delta\}$ and put $U=U_a\cup U_b$. Notice that if $z\in U\cap \{z:|z-a|\geq |z-b|\}$ then $|a-b|\leq |z-a|+|z-b|\leq 2|z-a|$ and hence $|z-a|\geq \delta/2$ for such $z$. Analogously, once $z\in U\cap \{z:|z-a|< |z-b|\}$, we have $|z-b|\geq \delta/2$. We first integrate over the set $U$ using all this:

\[
			\begin{split}
I_1&:=\int_{U}\left|\frac{1}{z-a}-\frac{1}{z-b}\right|dm(z) \\
&\leq  \delta \int_U\frac{dm(z)}{|z-a| |z-b|} \\
&\leq  2 \int_{U\cap \{|z-a|\geq |z-b|\}} \frac{dm(z)}{|z-b|} + 2\int_{U\cap \{|z-a|< |z-b| \}}\frac{dm(z)}{|z-a|}\\
&\leq 2 \int_{U_b} \frac{dm(z)}{|z-b|}+ 2\int_{U_a}\frac{dm(z)}{|z-a|}\lesssim \delta.
\end{split}
			\]

It remains to estimate the integral
\[
			\begin{split}
I_2&:=\int_{\D\backslash U}\left|\frac{1}{z-a}-\frac{1}{z-b}\right|dm(z) \\
&\leq \delta \int_{\D\backslash U} \frac{1}{|z-a|^2} + \frac{1}{|z-b|^2} dm(z)  \lesssim \delta\biggl( \log\Bigl(\frac{1}{\delta}\Bigr)+1 \biggr),
\end{split}
			\]
where we have used an integration in polar coordinates in the last estimate above. 

The Proposition hence follows,  since the integral that is of interest to us is equal to $I_1+I_2$.
\end{proof}
		\subsection{Open questions}
 The following questions remain open.

	First, we would like to ask whether the claim of our Theorem~\ref{thm:lowerbound} in two dimensions holds in the  case where the poles $z_k$ lie not only on the unit circle, but are also allowed to be inside the unit disc. We  know that this is true once the sum 
    $$\sum_{k=1}^n\mathrm{dist}(z_k,\mathbb T)\log\left(\frac{1}{\mathrm{dist}(z_k, \mathbb T)}\right)$$
    is small enough. Indeed, in this case one can push the poles $z_k$ out to the boundary $\mathbb T$ and reduce the needed estimate to the Newman bound~\eqref{eq:newbound}. The corresponding error can be controlled, thanks to  Proposition~\ref{lem1}. We also know this is true in case where the sum $\sum_{k=1}^n\mathrm{dist}(z_k,\mathbb T)$ is bounded from below by an absolute constant, as the following lemma gives.
    \begin{lemma}
       Suppose that $z_k\in \D$ for $k\in \{1,\ldots,n\}$. Then,
        $$2\pi\sum_{k=1}^n \mathrm{dist}(z_k,\mathbb T) \leq \int_{\mathbb D} \biggl|\sum_{k=1}^n \frac{1}{z-z_k}\biggr| dm(z).$$
    \end{lemma}
    \begin{proof}
        Fix $\sigma>0$ and denote $\phi_\sigma(z)=\max((1-\sigma-|z|),0)$, $z\in\overline{\mathbb D}$. Notice that the function $\phi_\sigma$ is $1$-Lipschitz and compactly supported in $\mathbb D$. 
        
        Let $\eta\in C^\infty(\mathbb C)$ be a standard mollifier such that $\mathrm{supp}(\eta)\subset\mathbb D$ and $\int_{\D} \eta=1$. For a fixed $\sigma>\epsilon>0$, put $\eta_\epsilon(z)=\epsilon^{-2}\eta(z/\epsilon)$, $z\in \mathbb C$.   

        Define  the following auxiliary function $\psi_{\sigma,\epsilon}=\phi_{\sigma}\ast\eta_\epsilon$. Notice that for the chosen $\sigma$ and $\epsilon$ we have $\mathrm{supp}(\psi_{\sigma,\epsilon})\subset \D$ and  $\psi_{\sigma,\epsilon}\in C^\infty(\mathbb C)$, so this is actually a test function. On the other hand, it holds  that $\|\nabla \psi_{\sigma,\epsilon} \|_{\infty}\leq 1$ since the mollification does not increase the Lipschitz constant.
        
 Observe that $$|\nabla u(z)|=\biggl|\sum_{k=1}^n \frac{1}{z-z_k}\biggr|,$$ where $u(z):=\sum_{k=1}^n \ln|z-z_k|$.
        It is also well known that $\Delta u=2\pi\sum_{k=1}^n \delta_{z_k}$ in the distributional sense. An integration by parts gives
        $$|\langle \Delta u, \psi_{\sigma,\epsilon}\rangle|\leq\int_{\mathbb{D}}|\nabla\psi_{\sigma,\epsilon}||\nabla u|\leq \int_{\mathbb{D}}|\nabla u|.$$
%
       %
        %
        But 
       $$ \langle \Delta u, \psi_{\sigma,\epsilon}\rangle =2\pi \sum_{k=1}^n \psi_{\sigma,\epsilon}(z_k).$$
        Now we first let $\epsilon$, and next $\sigma$ tend to zero to  see that the result follows.
    \end{proof}
	
	Second, as mentioned above, we  don't know whether our bound in Theorem~\ref{thm:lowerbound} is optimal in the case where the dimension of the ambient space is three and higher, even if the weights are all equal to $1$. 
	
Third, it is also interesting to find out if counterparts of	Theorems~\ref{thm:lowerbound} and~\ref{thm:upperbound_d=2}  hold for other energies such as $s$-Riesz energies.

Finally, we would like to know what is the correct formulation of Chui's weighted conjecture. We do not know the answer even in the two dimensional case.


\end{document}